\documentclass[11pt, final]{article}
\usepackage{a4}
\usepackage{amsmath}%
\usepackage{amstext}%
\usepackage{amssymb}%
\usepackage{showkeys}%
\usepackage{epsfig}%
\usepackage{cite}
\usepackage{tikz}
\usepackage{subfig}
\usepackage{caption}
\usepackage{multirow}
\usepackage{booktabs}
\usepackage{floatrow}

\usepackage{listings}
\newtheorem{theorem}{Theorem}

\newtheorem{axiom}{Axiom}

\newtheorem{definition}[axiom]{Definition}

\newtheorem{lemma}[theorem]{Lemma}

\newtheorem{rem}[theorem]{Remark}
\newenvironment{remark}{\rem\rm}{\endrem}

\newcounter{unnumber}

\newtheorem{prf}[unnumber]{Proof.}
\newenvironment{proof}{\prf\rm}{\hfill{$\blacksquare$}\endprf}
\newcommand{\R}{\mathbb{R}}%
\newcommand{\ol}{\overline}%
\newcommand{\ul}{\underline}%

\renewcommand{\>}{\right\rangle}
\newcommand{\<}{\left\langle}

\DeclareMathOperator*\gr{Gr}%
\DeclareMathOperator*\id{Id}%
\DeclareMathOperator*\prox{prox}%
\DeclareMathOperator*\argmin{argmin}

\DeclareMathOperator*\zer{zer}

\title{Convergence rates for forward-backward dynamical systems associated with strongly monotone inclusions}

\author{Radu Ioan Bo\c{t} \thanks{University of Vienna, Faculty of Mathematics, Oskar-Morgenstern-Platz 1, A-1090 Vienna, Austria,
email: radu.bot@univie.ac.at.} \and
Ern\"{o} Robert Csetnek \thanks {University of Vienna, Faculty of Mathematics, Oskar-Morgenstern-Platz 1, A-1090 Vienna, Austria,
email: ernoe.robert.csetnek@univie.ac.at. Research supported by FWF (Austrian Science Fund), Lise Meitner Programme, project M 1682-N25.}}

\begin{document}
\maketitle

\noindent \textbf{Abstract.} We investigate the convergence rates of the trajectories generated by implicit first and second order dynamical systems 
associated to the determination of the zeros of the sum of a maximally monotone operator and a monotone and Lipschitz continuous one in a real Hilbert space. 
We show that these trajectories strongly converge with exponential rate to a zero of the sum, provided the latter is strongly monotone. We derive from here convergence rates 
for the trajectories generated by dynamical systems associated to the minimization of the sum of a proper, convex and lower semicontinuous function with a smooth convex one provided 
the objective function fulfills a strong convexity assumption. In the particular case of minimizing a smooth and strongly convex function, we prove that its values converge along
the trajectory to its minimum value with exponential rate, too.\vspace{1ex}

\vspace{1ex}


\section{Introduction and preliminaries}\label{sec-intr}

The main topic of this paper is the investigation of convergence rates for implicit dynamical systems associated with monotone inclusion problems of the form
\begin{equation}\label{monpr}
\mbox{find} \ x^* \in {\cal H} \ \mbox{such that} \ 0 \in A x^* + B x^*, 
 \end{equation}
where ${\cal H}$ is a real Hilbert space, $A: {\cal H} \rightrightarrows {\cal H}$ is a maximally monotone operator, $B: {\cal H} \rightarrow {\cal H}$ is a monotone and $\frac{1}{\beta}$-Lipschitz continuous operator for $\beta > 0$ and 
$A+B$ is $\rho$-strongly monotone for $\rho > 0$. Dynamical systems of implicit type have been already considered in the literature in 
\cite{bolte-2003, abbas-att-arx14, att-sv2011, abbas-att-sv, antipin, att-alv-sv, b-c-dyn-KM, b-c-dyn-pen, b-c-dyn-sec-ord}.

We deal in a first instance with the first order dynamical system with variable relaxation parameters
\begin{equation}\label{f-o-dyn-syst-fb}\left\{
\begin{array}{ll}
\dot x(t)=\lambda(t)\left[J_{\eta A}\Big(x(t)-\eta B(x(t))\Big)-x(t)\right]\\
x(0)=x_0,
\end{array}\right.\end{equation}
where $x_0 \in {\cal H}$, $\lambda : [0,+\infty) \rightarrow [0,\infty)$ is a Lebesgue measurable function and $J_{\eta A}$ denotes the resolvent of the operator $\eta A$ for $\eta > 0$.

We notice that Abbas and Attouch considered in \cite[Section 4.2]{abbas-att-arx14} the dynamical system of same type
\begin{equation}\label{syst-abb-att}\left\{
\begin{array}{ll}
\dot x(t)+x(t)=\prox\nolimits_{\mu\Phi}\big(x(t)-\mu B(x(t))\big)\\
x(0)=x_0
\end{array}\right.\end{equation}
in connection to the determination of the zeros of $\partial \Phi + B$, where $\Phi:{\cal H}\rightarrow\R\cup\{+\infty\}$ is a proper, convex and lower semicontinuous function, 
$B:{\cal H}\rightarrow {\cal H}$ is a cocoercive operator, $\partial \Phi$ denotes the convex subdifferential of $\Phi$ and $\prox\nolimits_{\mu\Phi}$ denotes 
the proximal point operator of $\mu\Phi$.

Before that, Antipin in \cite{antipin} and Bolte in \cite{bolte-2003} studied the convergence of the trajectories generated by
\begin{equation}\label{syst-bolte}\left\{
\begin{array}{ll}
\dot x(t)+x(t)=P_C\big(x(t)-\mu\nabla g(x(t))\big)\\
x(0)=x_0
\end{array}\right.\end{equation}
to a minimizer of the smooth and convex function $g:{\cal H}\rightarrow\R$ over the nonempty, convex and closed set $C \subseteq {\cal H}$, where $\mu > 0$ and $P_C$ denotes the projection operator on the set $C$.

In the second part of the paper we approach the monotone inclusion \eqref{monpr} via the second order dynamical system with variable damping and relaxation parameters
\begin{equation}\label{s-o-dyn-syst-fb}\left\{
\begin{array}{ll}
\ddot x(t) + \gamma(t) \dot x(t) + \lambda(t)\left[x(t)-J_{\eta A}\Big(x(t)-\eta B(x(t))\Big)\right]=0\\
x(0)=u_0, \dot x(0)=v_0,
\end{array}\right.\end{equation}
where $u_0, v_0 \in {\cal H}$, $\lambda : [0,+\infty) \rightarrow [0,\infty)$ and $\gamma : [0,+\infty) \rightarrow [0,\infty)$ are Lebesgue measurable functions, and $\eta >0$.

Second order dynamical systems of the form
\begin{equation}\label{alv-att-nonexp-intr}\left\{
\begin{array}{ll}
\ddot x(t) + \gamma\dot x(t) + x(t)-Tx(t) = 0\\
x(0)=u_0, \dot x(0)=v_0,
\end{array}\right.\end{equation}
for $\gamma >0$ and $T:{\cal H} \rightarrow{\cal H}$ a nonexpansive operator, have been treated by Attouch and Alvarez in \cite{att-alv} in connection to the problem of approaching the fixed points of $T$.

For the minimization of the smooth and convex function $g:{\cal H}\rightarrow\R$ over the nonempty, convex and closed set $C \subseteq {\cal H}$, a continuous in time second order gradient-projection approach 
has been considered in \cite{att-alv, antipin}, having as starting point the dynamical system
\begin{equation}\label{gr-pr}\left\{
\begin{array}{ll}
\ddot x(t) + \gamma\dot x(t) + x(t)-P_C(x(t)-\eta\nabla g (x(t))) = 0\\
x(0)=u_0, \dot x(0)=v_0,
\end{array}\right.\end{equation}
with constant damping parameter $\gamma > 0$ and constant step size $\eta > 0$.

For an exhaustive asymptotic analysis of the first and second order dynamical systems \eqref{f-o-dyn-syst-fb} and \eqref{s-o-dyn-syst-fb}, in case $B$ is cocoercive, we refer the reader to \cite{b-c-dyn-KM} and \cite{b-c-dyn-sec-ord}, respectively.
According to the above-named works, one can expect under mild assumptions on the relaxation and, in the second order setting, on the damping functions, that the generated trajectories converge to a zero of $A+B$. The main scope 
of this paper is to show that when weakening the assumptions on $B$ to monotonicity and Lipschitz continuity, however, provided that $A+B$ is strongly monotone, the trajectories converge strongly to the unique
zero of $A+B$ with an exponential rate. Exponential convergence rates have been obtained also by Antipin in \cite{antipin} for the dynamical systems \eqref{syst-bolte} and \eqref{gr-pr}, by imposing for the 
smooth function $g$ supplementary strong convexity assumptions.

We transfer the results obtained for both first and second order dynamical systems to optimization problems of the form
\begin{equation}\label{convopt}
\min_{x \in {\cal H}} \ f(x) + g(x),
\end{equation}
where $f : {\cal H} \to \R\cup\{+\infty\}$ is a proper, convex and lower semicontinuous function, $g: {\cal H} \to \R$ is a convex  and (Fr\'echet) differentiable function with $\frac{1}{\beta}$-Lipschitz continuous gradient for $\beta > 0$ 
and $f+g$ is  $\rho$-strongly convex for $\rho > 0$, by taking into consideration that its set of minimizers coincides with the solution set of the monotone inclusion problem
\begin{equation*}
\mbox{find} \ x^* \in {\cal H} \ \mbox{such that} \ 0 \in \partial f(x^*) + \nabla g(x^*). 
 \end{equation*}
When further particularizing this context to the one of solving  minimization problems like
\begin{equation}\label{gopt}
\min_{x \in {\cal H}} \ g(x),
\end{equation}
where $g: {\cal H} \to \R$ is a $\rho$-strongly convex and (Fr\'echet) differentiable function with $\frac{1}{\beta}$-Lipschitz continuous gradient for $\rho > 0$ and $\beta > 0$, we show that the values of $g$ converge along the trajectories generated 
by the corresponding first and second order dynamical systems to its minimum value also with exponential rate.

The rest of this section is devoted to some notations and definitions used in the paper. We denote by ${\cal H}$ a real Hilbert space with inner product
$\langle\cdot,\cdot\rangle$ and corresponding norm $\|\cdot\|=\sqrt{\langle \cdot,\cdot\rangle}$. For an arbitrary set-valued operator $A:{\cal H}\rightrightarrows {\cal H}$ we denote by 
$\gr A=\{(x,u)\in {\cal H}\times {\cal H}:u\in Ax\}$ its graph. We use also the notation $\zer A=\{x\in{\cal{H}}:0\in Ax\}$ for the set of zeros of $A$. 
We say that $A$ is \emph{monotone}, if $\langle x-y,u-v\rangle\geq 0$ for all $(x,u),(y,v)\in\gr A$. A monotone operator $A$ is said to be \emph{maximally monotone}, if there exists no proper monotone extension of the graph of $A$ on 
${\cal H}\times {\cal H}$. The \emph{resolvent} of $A$, $J_A:{\cal H} \rightrightarrows {\cal H}$, is defined by $J_A=(\id +A)^{-1}$, where $\id : {\cal H} \rightarrow {\cal H}$ denotes the identity operator on ${\cal H}$. 
If $A$ is maximally monotone, then $J_A:{\cal H} \rightarrow {\cal H}$ is single-valued and 
maximally monotone (see \cite[Proposition 23.7 and Corollary 23.10]{bauschke-book}). For an arbitrary $\gamma>0$ we have (see \cite[Proposition 23.2]{bauschke-book})
\begin{equation}p\in J_{\gamma A}x \ \mbox{if and only if} \ (p,\gamma^{-1}(x-p))\in\gr A.\end{equation}
The operator $A$ is said to be \emph{$\rho$-strongly monotone} for $\rho > 0$, 
if $\langle x-y,u-v\rangle\geq \rho\|x-y\|^2$ for all $(x,u),(y,v)\in\gr A$. 

As in \cite{att-sv2011, abbas-att-sv}, we consider the following definition of an absolutely continuous function.

\begin{definition}\label{abs-cont} \rm (see, for instance, \cite{att-sv2011, abbas-att-sv}) A function $x:[0,b]\rightarrow {\cal H}$ (where $b>0$) is said to be absolutely continuous if one of the 
following equivalent properties holds:

(i)  there exists an integrable function $y:[0,b]\rightarrow {\cal H}$ such that $$x(t)=x(0)+\int_0^t y(s)ds \ \ \forall t\in[0,b];$$

(ii) $x$ is continuous and its distributional derivative is Lebesgue integrable on $[0,b]$; 

(iii) for every $\varepsilon > 0$, there exists $\eta >0$ such that for any finite family of intervals $I_k=(a_k,b_k) \subseteq [0,b]$ we have the implication
$$\left(I_k\cap I_j=\emptyset \mbox{ and }\sum_k|b_k-a_k| < \eta\right)\Longrightarrow \sum_k\|x(b_k)-x(a_k)\| < \varepsilon.$$
\end{definition}

\begin{remark}\label{rem-abs-cont}\rm (a) It follows from the definition that an absolutely continuous function is differentiable almost 
everywhere, its derivative coincides with its distributional derivative almost everywhere and one can recover the function from its derivative $\dot x=y$ by the integration formula (i). 

(b) If $x:[0,b]\rightarrow {\cal H}$, where $b>0$, is absolutely continuous and $B:{\cal H}\rightarrow {\cal H}$ is $L$-Lipschitz continuous for $L\geq 0$, then the function $z=B\circ x$ is absolutely continuous, too.
This can be easily seen by using the characterization of absolute continuity in
Definition \ref{abs-cont}(iii). Moreover, $z$ is differentiable almost everywhere on $[0,b]$ and the inequality $\|\dot z (t)\|\leq L\|\dot x(t)\|$ holds for almost every $t \in [0,b]$.   
\end{remark}

\section{Converges rates for first order dynamical systems}\label{sec2}

The starting point of the investigations we carry out in this section is the first order dynamical system \eqref{f-o-dyn-syst-fb} that we formulated in relation to the monotone inclusion problem \eqref{monpr}. 
We say that $x:[0,+\infty)\rightarrow{\cal H}$ is a \emph{strong global solution} of \eqref{f-o-dyn-syst-fb}, if the following properties are satisfied:

(i) $x:[0,+\infty)\rightarrow {\cal H}$ is \emph{locally absolutely continuous}, that is, absolutely continuous on each interval $[0,b]$ for $0<b<+\infty$; 

(ii) For almost every $t\in[0,+\infty)$ it holds $\dot x(t)=\lambda(t)\left[J_{\eta A}\Big(x(t)-\eta B(x(t))\Big)-x(t)\right]$;

(iii) $x(0)=x_0$.

The existence and uniqueness of strong global solutions of the system \eqref{f-o-dyn-syst-fb} follow from the Cauchy-Lipschitz-Picard Theorem, by noticing that the operator 
$T=J_{\eta A}\circ (\id-\eta B)-\id$ is Lipschitz continuous (see also \cite[Section 2]{b-c-dyn-KM}). 

The following result can bee seen as the continuous counterpart of \cite[Proposition 25.9]{bauschke-book}, where it is shown that the sequence iteratively generated by the 
forward-backward algorithm linearly converges to the unique solution of \eqref{monpr}, provided that one of the two involved operators is strongly monotone.

\begin{theorem}\label{first-order-a+b} Let $A:{\cal H}\rightrightarrows {\cal H}$ be a maximally monotone operator,
$B:{\cal H}\rightarrow {\cal H}$ a monotone and $\frac{1}{\beta}$-Lipschitz continuous operator for $\beta > 0$ such that $A+B$ is $\rho$-strongly monotone for $\rho>0$ and $x^*$ be the unique point in $\zer(A+B)$.
Let $\lambda:[0,+\infty)\rightarrow[0,+\infty)$ be a Lebesgue measurable function such that there exist real numbers
$\ul\lambda$ and $\ol\lambda$ fulfilling $$0<\ul\lambda\leq\inf_{t\geq 0}\lambda(t)\leq \sup_{t\geq 0}\lambda(t)\leq\ol\lambda.$$
Chose $\alpha >0$ and $\eta > 0$ such that
$$\alpha <2\rho\beta^2\ul\lambda \ \mbox{and} \  \frac{1}{\beta} + \frac{\ol\lambda}{2\alpha} \leq \rho+ \frac{1}{\eta}.$$
If $x_0\in {\cal H}$ and $x:[0,+\infty)\rightarrow{\cal H}$ is the unique strong global solution of the dynamical system \eqref{f-o-dyn-syst-fb}, then
for every $t \in [0,+\infty)$ one has $$\|x(t)-x^*\|^2\leq \|x_0-x^*\|^2\exp(-Ct),$$
where $C:=\frac{2\rho\ul\lambda-\frac{\alpha}{\beta^2}}{2\rho+\frac{1}{\eta}}>0$.  
\end{theorem}

\begin{proof} Notice that $B$ is a maximally monotone operator (see \cite[Corollary 20.25]{bauschke-book}) and, since $B$ has full 
domain, $A+B$ is maximally monotone, too (see \cite[Corollary 24.4]{bauschke-book}). Therefore, due to the strong monotonicity of $A+B$, $\zer(A+B)$ is a singleton (see \cite[Corollary 23.37]{bauschke-book}). 

A direct consequence of \eqref{f-o-dyn-syst-fb} and of the definition of the resolvent is the inclusion 
$$-\frac{1}{\eta\lambda(t)}\dot x(t)-B(x(t))+B\left(\frac{1}{\lambda(t)}\dot x(t)+x(t)\right)\in (A+B)\left(\frac{1}{\lambda(t)}\dot x(t)+x(t)\right),$$
which holds for almost every $t\in[0,+\infty)$. Combining it with $0\in(A+B)(x^*)$ and the strong monotonicity of $A+B$, it yields for almost every $t \in [0,+\infty)$
\begin{align*}
& \rho\left\|\frac{1}{\lambda(t)}\dot x(t)+x(t)-x^*\right\|^2\leq \\
& \<\frac{1}{\lambda(t)}\dot x(t)+x(t)-x^*, -\frac{1}{\eta\lambda(t)}\dot x(t)-B(x(t))+B\left(\frac{1}{\lambda(t)}\dot x(t)+x(t)\right)\>.
\end{align*}

By using the notation $h(t)=\frac{1}{2}\|x(t)-x^*\|^2$ for $t \in [0,+\infty)$, the Cauchy-Schwartz inequality, the Lipschitz property of $B$ and the fact that $\dot h(t) = \langle x(t) - x^*, \dot x(t) \rangle$, 
we deduce that for almost every $t \in [0,+\infty)$
\begin{align*}
\rho\left\|\frac{1}{\lambda(t)}\dot x(t)+x(t)-x^*\right\|^2 \!\!\!\leq  &  \!-\frac{1}{\eta\lambda^2(t)}\|\dot x(t)\|^2+\frac{1}{\lambda(t)}\!\<\dot x(t),B\!\left(\frac{1}{\lambda(t)}\dot x(t)+x(t)\right)-B(x(t))\>\\
& \!-\frac{1}{\eta\lambda(t)}\dot h(t)+\<x(t)-x^* , B\left(\frac{1}{\lambda(t)}\dot x(t)+x(t)\right)-B(x(t))\> \\
\leq & -\frac{1}{\eta\lambda^2(t)}\|\dot x(t)\|^2+\frac{1}{\beta\lambda^2(t)}\|\dot x(t)\|^2-\frac{1}{\eta\lambda(t)}\dot h(t)\\
& \! + \frac{1}{\beta\lambda(t)} \|x(t) - x^*\| \|\dot x(t) \|\\
\leq & -\frac{1}{\eta\lambda^2(t)}\|\dot x(t)\|^2+\frac{1}{\beta\lambda^2(t)}\|\dot x(t)\|^2-\frac{1}{\eta\lambda(t)}\dot h(t)\\
& \! +\frac{\alpha}{\beta^2\lambda(t)}h(t) +\frac{1}{2\alpha\lambda(t)}\|\dot x(t)\|^2.
\end{align*}
As 
$$\rho\left\|\frac{1}{\lambda(t)}\dot x(t)+x(t)-x^*\right\|^2=\frac{\rho}{\lambda^2(t)}\|\dot x(t)\|^2+\frac{2\rho}{\lambda(t)}\dot h(t)+2\rho h(t),$$
we obtain for almost every $t \in [0,+\infty)$ the inequality
\begin{align*}
& \left(\frac{2\rho}{\lambda(t)}+\frac{1}{\eta\lambda(t)}\right)\dot h(t)+\left(2\rho-\frac{\alpha}{\beta^2\lambda(t)}\right)h(t) +\\
& \left(\frac{\rho}{\lambda^2(t)}+\frac{1}{\eta\lambda^2(t)}-\frac{1}{\beta\lambda^2(t)}-\frac{1}{2\alpha\lambda(t)}\right)\|\dot x(t)\|^2\leq 0.
\end{align*}
However, the way in which the involved parameters were chosen imply for almost every $t\in[0,+\infty)$ that
$$\left(\frac{2\rho}{\lambda(t)}+\frac{1}{\eta\lambda(t)}\right)\dot h(t)+\left(2\rho-\frac{\alpha}{\beta^2\lambda(t)}\right)h(t)\leq 0$$
or, equivalently,
$$\dot h(t)+\frac{2\rho\lambda(t)-\frac{\alpha}{\beta^2}}{2\rho+\frac{1}{\eta}}h(t)\leq 0.$$
This further implies 
$$\dot h(t)+Ch(t)\leq 0$$
for almost every $t\in[0,+\infty)$. By multiplying this inequality with $\exp(Ct)$ and integrating from $0$ to $T$, where $T\geq 0$, one easily obtains the conclusion.  
\end{proof}

We come now to the convex optimization problem \eqref{convopt} and notice that, 
since $\argmin (f+g) = \zer (\partial (f+g)) = \zer (\partial f + \nabla g)$, one can approach this set by means of the trajectories of the dynamical system \eqref{f-o-dyn-syst-fb} written for $A=\partial f$ and $B = \nabla g$. Here, 
$\partial f : {\cal H} \rightrightarrows {\cal H}$, defined by 
$$\partial f(x)=\{u\in {\cal H}:f(y)\geq f(x)+\<u,y-x\> \ \forall y\in {\cal H}\},$$
if $f(x) \in \R$ and $\partial f(x) = \emptyset$, otherwise, denotes the \emph{convex subdifferential} of $f$, which is a maximally monotone operator, provided that $f$ is proper, convex and 
lower semicontinuous (see \cite{rock}). We notice that, for $\eta > 0$, 
the resolvent of $\eta \partial f$ is given by $J_{\eta \partial f}=\prox_{\eta f}$ (see \cite{bauschke-book}),
where $\prox_{\eta f}:{\cal H}\rightarrow {\cal H}$,
\begin{equation}\label{prox-def}\prox\nolimits_{\eta f}(x)=\argmin_{y\in {\cal H}}\left \{f(y)+\frac{1}{2\eta}\|y-x\|^2\right\},
\end{equation}
denotes the \emph{proximal point operator} of $\eta f$. This being said, the dynamical system \eqref{f-o-dyn-syst-fb} becomes
\begin{equation}\label{f-o-dyn-syst-fb-opt}\left\{
\begin{array}{ll}
\dot x(t)=\lambda(t)\left[\prox_{\eta f}\Big(x(t)-\eta \nabla g(x(t))\Big)-x(t)\right]\\
x(0)=x_0.
\end{array}\right.\end{equation}
The following result is a direct consequence of Theorem \ref{first-order-a+b}. Let us also notice that $f+g$ is said to be $\rho$-strongly convex for $\rho>0$, 
if $f + g -\frac{\rho}{2}\|\cdot\|^2$ is a convex function. In this situation $\partial(f+g) = \partial f + \nabla g$ is a $\rho$-strongly monotone operator (see \cite[Example 22.3(iv)]{bauschke-book}.)

\begin{theorem}\label{first-order-a+b-opt} Let $f:{\cal H}\rightarrow \R\cup\{+\infty\}$ be a proper, convex 
and lower semicontinuous function, $g:{\cal H}\rightarrow \R$ be a convex and (Fr\'{e}chet) differentiable function with $\frac{1}{\beta}$-Lipschitz continuous gradient for $\beta > 0$ such that 
$f+g$ is $\rho$-strongly convex for $\rho>0$ and  $x^*$ be the unique minimizer of $f+g$ over ${\cal H}$. Let $\lambda:[0,+\infty)\rightarrow[0,+\infty)$ be a Lebesgue measurable function such that there exist real numbers
$\ul\lambda$ and $\ol\lambda$ fulfilling $$0<\ul\lambda\leq\inf_{t\geq 0}\lambda(t)\leq \sup_{t\geq 0}\lambda(t)\leq\ol\lambda.$$
Chose $\alpha >0$ and $\eta > 0$ such that
$$\alpha <2\rho\beta^2\ul\lambda \ \mbox{and} \  \frac{1}{\beta} + \frac{\ol\lambda}{2\alpha} \leq \rho+ \frac{1}{\eta}.$$
If $x_0\in {\cal H}$ and $x:[0,+\infty)\rightarrow{\cal H}$ is the unique strong global solution of the dynamical system \eqref{f-o-dyn-syst-fb-opt}, then
for every $t \in [0.+\infty)$ one has $$\|x(t)-x^*\|^2\leq \|x_0-x^*\|^2\exp(-Ct),$$
where $C:=\frac{2\rho\ul\lambda-\frac{\alpha}{\beta^2}}{2\rho+\frac{1}{\eta}}>0$.  
\end{theorem}

In the last part of this section we approach the convex minimization problem \eqref{gopt} via the first order dynamical system
\begin{equation}\label{f-o-nabla}\left\{
\begin{array}{ll}
\dot x(t) + \lambda(t)\nabla g(x(t))=0\\
x(0)=x_0.
\end{array}\right.\end{equation}
The following result quantifies the rate of convergence of $g$  to its minimum value along the trajectories generated by \eqref{f-o-nabla}.

\begin{theorem}\label{first-order-nabla} Let $g:{\cal H}\rightarrow \R$ be a $\rho$-strongly convex and 
(Fr\'echet) differentiable function with $\frac{1}{\beta}$-Lipschitz continuous gradient for $\rho > 0$ and $\beta > 0$ and $x^*$ be the unique minimizer of $g$ over ${\cal H}$.
Let $\lambda:[0,+\infty)\rightarrow[0,+\infty)$ be a Lebesgue measurable function such that there exists a real number $\ul\lambda\in\R$ fulfilling $$0<\ul\lambda\leq\inf_{t\geq 0}\lambda(t).$$
Chose $\alpha>0$ such that 
$$\alpha \leq 2\ul\lambda\beta\rho^2.$$
If $x_0\in {\cal H}$ and $x:[0,+\infty)\rightarrow{\cal H}$ is the unique strong global solution of the dynamical system \eqref{f-o-nabla}, then
for every $t \in [0, +\infty)$ one has
$$0 \leq \frac{\rho}{2} \|x(t) - x^*\|^2 \leq g(x(t))-g(x^*)\leq (g(x_0)-g(x^*))\exp(-\alpha t)\leq\frac{1}{2\beta}\|x_0-x^*\|^2\exp(-\alpha t).$$  
\end{theorem}

\begin{proof} The second inequality is a consequence of the strong convexity of the function $g$. 
Further, we recall that according to the descent lemma, which is valid for an arbitrary differentiable function with Lipschitz continuous gradient (see \cite[Lemma 1.2.3]{nes}), we have
$$g(u)\leq g(v)+\langle \nabla g(v),u-v\rangle+\frac{1}{2\beta}\|u-v\|^2 \ \forall u,v \in {\cal H}.$$
By setting in the previous relation, for every $t \in [0, +\infty)$, $u:= x(t)$ and $v:=x^*$ and by taking into account that $\nabla g(x^*) = 0$, we obtain
\begin{equation}\label{desc-l}g(x(t))-g(x^*)\leq \frac{1}{2\beta}\|x(t)-x^*\|^2.\end{equation}
From here, the last inequality in the conclusion follows automatically.

Using the strong convexity of $g$ we have for every $t \in [0, +\infty)$ that
$$\rho\|x(t)-x^*\|^2\leq \<x(t)-x^*,\nabla g(x(t))\>\leq\|x(t)-x^*\| \|\nabla g(x(t))\|,$$
thus 
\begin{equation}\label{ineq-g-str-conv} \rho\|x(t)-x^*\|\leq\|\nabla g(x(t))\|.\end{equation}
Finally, from the first equation in \eqref{f-o-nabla}, \eqref{desc-l}, \eqref{ineq-g-str-conv} and using the way in which $\alpha$ was chosen, we obtain for almost every $t \in [0,+\infty)$
\begin{align*} \frac{d}{dt}\big(g(x(t))-g(x^*)\big)+\alpha(g(x(t))-g(x^*))= & \<\dot x(t) , \nabla g(x(t))\>+\alpha(g(x(t))-g(x^*))\\
\leq &  -\lambda(t)\|\nabla g(x(t))\|^2+\frac{\alpha}{2\beta}\|x(t)-x^*\|^2\\
\leq & \left(-\lambda(t)+\frac{\alpha}{2\beta\rho^2}\right)\|\nabla g(x(t))\|^2\\
\leq & 0.
\end{align*}
By multiplying this inequality with $\exp(\alpha t)$ and integrating from $0$ to $T$, where $T\geq 0$, one easily obtains also the third inequality.  
\end{proof}

\section{Converges rates for second order dynamical systems}\label{sec3}

The starting point of the investigations we go through in this section is again the monotone inclusion problem \eqref{monpr}, however, this time approached via the second order dynamical system
\eqref{s-o-dyn-syst-fb}. We say that $x:[0,+\infty)\rightarrow{\cal H}$ is a \emph{strong global solution} of \eqref{s-o-dyn-syst-fb}, if the following properties are satisfied:

(i) $x, \dot x:[0,+\infty)\rightarrow {\cal H}$ are locally absolutely continuous; 

(ii) For almost every $t\in[0,+\infty)$ it holds 
$$\ddot x(t) + \gamma(t) \dot x(t) + \lambda(t)\left[x(t)-J_{\eta A}\Big(x(t)-\eta B(x(t))\Big)\right]=0;$$

(iii) $x(0)=u_0, \dot x(0)=v_0$. 

The existence and uniqueness of strong global solutions of the system \eqref{s-o-dyn-syst-fb} follow from the Cauchy-Lipschitz-Picard Theorem applied in a product space (see also \cite{b-c-dyn-sec-ord}). 

The following result will be useful when deriving the convergence rates.

\begin{lemma}\label{l-sec-ord} Let $h,\gamma,b_1,b_2,b_3,u:[0,+\infty)\rightarrow\R$ be given functions such that  $h,\gamma,b_2,u$ are locally absolutely continuous and    
$\dot h$ is locally absolutely continuous, too. Assume that 
$$h(t),b_2(t),u(t)\geq 0 \ \forall t \in [0,+\infty)$$
and that there exists $\ul\gamma>1$ such that 
$$\gamma(t)\geq\ul\gamma>1 \ \forall t \in [0,+\infty).$$
Further, assume that for almost every $t\in[0,+\infty)$ one has 
\begin{align}\label{hyp1}
\gamma(t)+\dot\gamma(t)\leq b_1(t)+1,
\end{align}
\begin{align}\label{hyp2}
b_2(t)+\dot b_2(t)\leq b_3(t)
\end{align}
and
\begin{equation}\label{ineq-h-u}\ddot h(t)+\gamma(t)\dot h(t)+b_1(t) h(t)+b_2(t)\dot u(t)+b_3(t)u(t)\leq 0.\end{equation}
Then there exists $M>0$ such that the following statements hold: 
\begin{enumerate}
 \item[(i)] if $1<\ul\gamma<2$, then for almost every $t\in[0,+\infty)$ 
 $$0\leq h(t)\leq \left(h(0)+\frac{M}{2-\ul\gamma}\right)\exp(-(\ul\gamma-1)t);$$
 \item[(ii)] if $2<\ul\gamma$, then for almost every $t\in[0,+\infty)$ 
 $$0\leq h(t)\leq h(0)\exp(-(\ul\gamma-1)t)+\frac{M}{\ul\gamma-2}\exp(-t)\leq \left(h(0)+\frac{M}{\ul\gamma-2}\right)\exp(-t);$$
 \item[(iii)] if $\ul\gamma=2$, then for almost every $t\in[0,+\infty)$ $$0\leq h(t)\leq \left(h(0)+Mt\right)\exp(-t).$$
\end{enumerate} 
\end{lemma}

\begin{proof} We multiply the inequality \eqref{ineq-h-u} with $\exp(t)$ and use the identities 
\begin{align*}
\exp(t)\ddot h(t) = & \frac{d}{dt}\big(\exp(t)\dot h(t)\big)-\exp(t)\dot h(t)\\
\exp(t)\dot u(t) = & \frac{d}{dt}(\exp(t)u(t))-\exp(t)u(t)\\
\exp(t)\dot h(t) = & \frac{d}{dt}\big(\exp(t) h(t)\big)-\exp(t) h(t)
\end{align*}
in order to derive for almost every $t\in[0,+\infty)$ the inequality
\begin{align*}& \frac{d}{dt}\big(\exp(t)\dot h(t)\big)+(\gamma(t)-1)\frac{d}{dt}\big(\exp(t)h(t)\big) + \\
&\exp(t)h(t)(b_1(t)+1-\gamma(t))+ b_2(t)\frac{d}{dt}\big(\exp(t)u(t)\big) +(b_3(t)-b_2(t))\exp(t)u(t)\leq 0.
\end{align*}
By using also 
\begin{align*}
(\gamma(t)-1)\frac{d}{dt}\big(\exp(t)h(t)\big) & =\frac{d}{dt}\Big((\gamma(t)-1)\exp(t)h(t)\Big)-\dot\gamma(t)\exp(t)h(t)\\
b_2(t)\frac{d}{dt}\big(\exp(t)u(t)\big) & =\frac{d}{dt}\big(b_2(t)\exp(t)u(t)\big)-\dot b_2(t)\exp(t)u(t)
\end{align*}
we obtain for almost every $t\in[0,+\infty)$
\begin{align*}& \frac{d}{dt}\big(\exp(t)\dot h(t)\big)+\frac{d}{dt}\Big((\gamma(t)-1)\exp(t)h(t)\Big)+\frac{d}{dt}\big(b_2(t)\exp(t)u(t)\big)+\\
&\big(b_1(t)+1-\gamma(t)-\dot\gamma(t)\big)\exp(t)h(t)+\big(b_3(t)-b_2(t)-\dot b_2(t)\big)\exp(t)u(t)\leq 0.
\end{align*}
The hypotheses regarding the parameters involved imply that the function 
$$t\rightarrow\exp(t)\dot h(t)+(\gamma(t)-1)\exp(t)h(t)+b_2(t)\exp(t)u(t)$$ is 
monotonically decreasing, hence there exists $M>0$ such that $$\exp(t)\dot h(t)+(\gamma(t)-1)\exp(t)h(t)+b_2(t)\exp(t)u(t)\leq M.$$

Since $u(t),b_2(t)\geq 0$ we get $$\dot h(t)+(\gamma(t)-1)h(t)\leq M\exp(-t),$$
hence $$\dot h(t)+(\ul\gamma-1)h(t)\leq M\exp(-t)$$
for every $t \in [0,+\infty)$. This implies that 
$$\frac{d}{dt}\big(\exp((\ul\gamma-1)t)h(t)\big)\leq M\exp((\ul\gamma-2)t),$$
for every $t \in [0,+\infty)$, from which the conclusion follows easily by integration. 
\end{proof}

We come now to the first main result of this section.

\begin{theorem}\label{sec-ord-a+b} Let $A:{\cal H}\rightrightarrows {\cal H}$ be a maximally monotone operator, $B:{\cal H}\rightarrow {\cal H}$ a monotone and $\frac{1}{\beta}$-Lipschitz continuous operator for $\beta > 0$ 
such that $A+B$ is $\rho$-strongly monotone for $\rho>0$ and $x^*$ be the unique point in $\zer(A+B)$. Chose $\alpha,\delta\in(0,1)$ and $\eta > 0$ such that $\delta\beta\rho<1$ and 
$\frac{1}{\eta}=\left(\frac{1}{\beta}+\frac{1}{4\rho\beta^2\alpha}\right)\frac{1}{\delta} - \rho >0$. 

Let  $\lambda:[0,+\infty)\rightarrow[0,+\infty)$ be a locally absolutely continuous 
function fulfilling for every $t \in [0,+\infty)$
\begin{align*}\theta(t):= & \ \lambda(t)\frac{\delta}{1-\delta}\frac{\rho+\left(\frac{1}{\beta}+\frac{1}{4\rho\beta^2\alpha}\right)\frac{1}{\delta}}{\frac{1}{\beta}+\frac{1}{4\rho\beta^2\alpha}}\\
\leq & \ \lambda(t)\frac{2\rho(1-\alpha)}{\rho+\left(\frac{1}{\beta}+\frac{1}{4\rho\beta^2\alpha}\right)\frac{1}{\delta}}
+\lambda^2(t)\left(\frac{2\rho(1-\alpha)}{\rho+\left(\frac{1}{\beta}+\frac{1}{4\rho\beta^2\alpha}\right)\frac{1}{\delta}}\right)^2
\end{align*}
and such that there exists a real number $\ul\lambda$ with the property that
$$0<\ul\lambda\leq\inf_{t\geq0}\lambda(t)$$ and 
$$2<\theta:=\ul\lambda\frac{\delta}{1-\delta}\frac{\rho+\left(\frac{1}{\beta}+\frac{1}{4\rho\beta^2\alpha}\right)\frac{1}{\delta}}{\frac{1}{\beta}+\frac{1}{4\rho\beta^2\alpha}}.$$

Further, let $\gamma:[0,+\infty)\rightarrow[0,+\infty)$ be a locally absolutely continuous 
function fulfilling 
\begin{equation}\label{th6cond1}
\frac{1+\sqrt{1+4\theta(t)}}{2}\leq\gamma(t)\leq 1+\lambda(t)\frac{2\rho(1-\alpha)}{\rho+\left(\frac{1}{\beta}+\frac{1}{4\rho\beta^2\alpha}\right)\frac{1}{\delta}} \ \mbox{for every} \ t \in [0,+\infty)
\end{equation}
and
\begin{equation}\label{th6cond2}
\dot\gamma(t)\leq 0 \mbox{ and } \frac{d}{dt}\left(\frac{\gamma(t)}{\lambda(t)}\right)\leq 0 \ \mbox{for almost every} \ t \in [0,+\infty).
\end{equation}
Let $u_0,v_0\in {\cal H}$  and $x:[0,+\infty)\rightarrow{\cal H}$ be the unique strong global solution of the dynamical system \eqref{s-o-dyn-syst-fb}.

Then $\gamma(t)\geq\ul\gamma:=\frac{1+\sqrt{1+4\theta}}{2}>2$ for every $t \in [0,+\infty)$ and there exists $M>0$ such that for every $t\in[0,+\infty)$
\begin{align*}
0\leq \|x(t)-x^*\|^2 & \leq \|u_0-x^*\|^2\exp(-(\ul\gamma-1)t)+\frac{M}{\ul\gamma-2}\exp(-t)\\
 & \leq \left(\|u_0-x^*\|^2+\frac{M}{\ul\gamma-2}\right)\exp(-t).
\end{align*}
\end{theorem}

\begin{proof} From the definition of the resolvent we have for almost every $t \in [0, +\infty)$
\begin{align}\label{conseq-def-res}
& B\left(\frac{1}{\lambda(t)}\ddot x(t)+\frac{\gamma(t)}{\lambda(t)}\dot x(t)+x(t)\right)-B(x(t))
-\frac{1}{\eta\lambda(t)}\ddot x(t)-\frac{\gamma(t)}{\eta\lambda(t)}\dot x(t)\in \nonumber \\ 
& (A+B)\left(\frac{1}{\lambda(t)}\ddot x(t)+\frac{\gamma(t)}{\lambda(t)}\dot x(t)+x(t)\right).\end{align}
We combine this with $0\in (A+B)x^*$, the strong monotonicity of $A+B$, the Lipschitz continuity of $B$ and, by also using the Cauchy-Schwartz inequality, we get for almost every $t \in [0, +\infty)$
\begin{align*} & \frac{\rho}{\lambda^2(t)}\left\|\ddot x(t)+\gamma(t)\dot x(t)\right\|^2+\frac{2\rho}{\lambda(t)}\<x(t)-x^*,\ddot x(t)+\gamma(t)\dot x(t)\> 
+\rho\|x(t)-x^*\|^2\\
= & \rho\left\|\frac{1}{\lambda(t)}\ddot x(t)+\frac{\gamma(t)}{\lambda(t)}\dot x(t)+x(t)-x^*\right\|^2\\ 
\leq & \< \frac{1}{\lambda(t)}\ddot x(t)+\frac{\gamma(t)}{\lambda(t)}\dot x(t)+x(t)-x^*,
B\left(\frac{1}{\lambda(t)}\ddot x(t)+\frac{\gamma(t)}{\lambda(t)}\dot x(t)+x(t)\right)-B(x(t))\>\\
& -  \< \frac{1}{\lambda(t)}\ddot x(t)+\frac{\gamma(t)}{\lambda(t)}\dot x(t)+x(t)-x^*, \frac{1}{\eta\lambda(t)}\ddot x(t) + \frac{\gamma(t)}{\eta\lambda(t)}\dot x(t)\>\\
= & \frac{1}{\lambda(t)}\< \ddot x(t)+\gamma(t)\dot x(t), B\left(\frac{1}{\lambda(t)}\ddot x(t)+\frac{\gamma(t)}{\lambda(t)}\dot x(t)+x(t)\right)-B(x(t))\>\\
& + \< x(t)-x^*, B\left(\frac{1}{\lambda(t)}\ddot x(t)+\frac{\gamma(t)}{\lambda(t)}\dot x(t)+x(t)\right)-B(x(t))\>\\
& -\frac{1}{\eta\lambda^2(t)}\|\ddot x(t)+\gamma(t)\dot x(t)\|^2 -\frac{1}{\eta\lambda(t)}\<x(t)-x^*,\ddot x(t)+\gamma(t)\dot x(t)\>\\
\leq &  \frac{1}{\beta\lambda^2(t)}\|\ddot x(t)+\gamma(t)\dot x(t)\|^2-\frac{1}{\eta\lambda^2(t)}\|\ddot x(t)+\gamma(t)\dot x(t)\|^2\\
&+\frac{1}{4\rho\beta^2 \alpha \lambda^2(t)}\!\|\ddot x(t)+\gamma(t)\dot x(t)\|^2+\rho\alpha \|x(t)-x^*\|^2 -\!\frac{1}{\eta\lambda(t)}\<x(t)-x^*,\ddot x(t)+\gamma(t)\dot x(t)\>.
\end{align*}
Using again the notation $h(t)=\frac{1}{2}\|x(t)-x^*\|^2$, we have for almost every $t \in [0,+\infty)$
\begin{equation}\label{norm^2}\left\|\ddot x(t)+\gamma(t)\dot x(t)\right\|^2=\|\ddot x(t)\|^2+\gamma^2(t)\|\dot x(t)\|^2+\gamma(t)\frac{d}{dt}(\|\dot x(t)\|^2)\end{equation}
and $$\<x(t)-x^*,\ddot x(t)+\gamma(t)\dot x(t)\>=\ddot h(t)+\gamma(t) \dot h(t)-\|\dot x(t)\|^2.$$
Therefore, we obtain for almost every $t \in [0,+\infty)$
\begin{align*} & \ \left(\frac{\rho}{\lambda^2(t)}+\frac{1}{\eta\lambda^2(t)}-\frac{1}{\beta\lambda^2(t)}-\frac{1}{4\rho\beta^2 \alpha \lambda^2(t)}\right)\|\ddot x(t)\|^2\\
+ & \ \left[\gamma^2(t)\left(\frac{\rho}{\lambda^2(t)}+\frac{1}{\eta\lambda^2(t)}-\frac{1}{\beta\lambda^2(t)}-\frac{1}{4\rho\beta^2 \alpha \lambda^2(t)}\right)-\frac{2\rho}{\lambda(t)}-\frac{1}{\eta\lambda(t)}\right]\|\dot x(t)\|^2\\
+ & \ \gamma(t)\left(\frac{\rho}{\lambda^2(t)}+\frac{1}{\eta\lambda^2(t)}-\frac{1}{\beta\lambda^2(t)}-\frac{1}{4\rho\beta^2 \alpha \lambda^2(t)}\right)\frac{d}{dt}\left(\|\dot x(t)\|^2\right)\\
+ & \ \left(\frac{2\rho}{\lambda(t)}+\frac{1}{\eta\lambda(t)}\right)\ddot h(t)+ \gamma(t)\left(\frac{2\rho}{\lambda(t)}+\frac{1}{\eta\lambda(t)}\right)\dot h(t)
+ 2\rho(1-\alpha)h(t) \leq 0. 
\end{align*}
The hypotheses imply that 
$$\frac{\rho}{\lambda^2(t)}+\frac{1}{\eta\lambda^2(t)}-\frac{1}{\beta\lambda^2(t)}-\frac{1}{4\rho\beta^2 \alpha \lambda^2(t)}
= \frac{1}{\lambda^2(t)}\left(\rho+\frac{1}{\eta}-\frac{1}{\beta}-\frac{1}{4\rho\beta^2\alpha}\right)>0,$$
hence the first term in the left hand side of the above inequality can be 
neglected and we obtain for almost every $t \in [0,+\infty)$ that
\begin{equation}\ddot h(t)+\gamma(t)\dot h(t)+b_1(t) h(t)+b_2(t)\frac{d}{dt}(\|\dot x(t)\|^2)+b_3(t)\|\dot x(t)\|^2\leq 0,\end{equation}
where $$b_1(t):=\lambda(t)\frac{2\rho(1-\alpha)}{2\rho+\frac{1}{\eta}}>0$$
$$b_2(t):=\gamma(t)\frac{\frac{\rho}{\lambda^2(t)}+\frac{1}{\eta\lambda^2(t)}-\frac{1}{\beta\lambda^2(t)}-\frac{1}{4\rho\beta^2 \alpha \lambda^2(t)}}{\frac{2\rho}{\lambda(t)}+\frac{1}{\eta\lambda(t)}}
=\frac{\gamma(t)}{\lambda(t)}\frac{\rho+\frac{1}{\eta}-\frac{1}{\beta}-\frac{1}{4\rho\beta^2\alpha}}{2\rho+\frac{1}{\eta}}>0$$ and 
$$b_3(t):=\frac{\gamma^2(t)\left(\frac{\rho}{\lambda^2(t)}+\frac{1}{\eta\lambda^2(t)}-\frac{1}{\beta\lambda^2(t)}-\frac{1}{4\rho\beta^2 \alpha \lambda^2(t)}\right)-\frac{2\rho}{\lambda(t)}-\frac{1}{\eta\lambda(t)}}
{\frac{2\rho}{\lambda(t)}+\frac{1}{\eta\lambda(t)}}.$$
This shows that \eqref{ineq-h-u} in Lemma \ref{l-sec-ord} for $u := \|\dot x(\cdot)\|^2$ is fulfilled. In order to apply Lemma \ref{l-sec-ord}, we have only to prove that \eqref{hyp1} and \eqref{hyp2} are satisfied, as every other assumption in this 
statement is obviously guaranteed.

A simple calculation shows that 
\begin{equation}\label{b3-b2}b_3(t)\geq b_2(t) \Longleftrightarrow
\gamma^2(t)-\gamma(t)\geq \frac{\frac{2\rho}{\lambda(t)}+\frac{1}{\eta\lambda(t)}}
{\frac{\rho}{\lambda^2(t)}+\frac{1}{\eta\lambda^2(t)}-\frac{1}{\beta\lambda^2(t)}-\frac{1}{4\rho\beta^2 \alpha \lambda^2(t)}} = \theta(t),\end{equation}
which is true according to \eqref{th6cond1}, thus $b_3(t)\geq b_2(t)$ for every $t \in [0,+\infty)$. On the other hand (see \eqref{th6cond2}), 
$$\dot b_2(t)\leq 0$$
for almost every $t \in [0,+\infty)$, from which \eqref{hyp2} follows.

Further, again by using \eqref{th6cond1}, observe that 
$$1+b_1(t)=1+\lambda(t)\frac{2\rho(1-\alpha)}{\rho+\left(\frac{1}{\beta}+\frac{1}{4\rho\beta^2\alpha}\right)\frac{1}{\delta}}\geq\gamma(t)$$
for every $t \in [0,+\infty)$, which, combined with 
$$\dot \gamma(t) \leq 0$$
for almost every $t \in [0,+\infty)$, shows that \eqref{hyp1} is also fulfilled. 

The conclusion follows from Lemma \ref{l-sec-ord}(ii), by noticing that $\ul\gamma>2$, as $\theta > 2$.
\end{proof}

\begin{remark}\label{rem31}
One can notice that when $\dot \gamma(t) \leq 0$ for almost every $t \in [0,+\infty)$, the second assumption in \eqref{th6cond2} is fulfilled provided that $\dot \lambda(t) \geq 0$ for almost every $t \in [0,+\infty)$.

Further, we would like to point out that one can oviously chose $\lambda(t) = \ul \lambda$ and $\gamma(t) = \gamma$ for every $t \in [0,+\infty)$, where
\begin{align*}2 < \theta := & \ \ul \lambda \frac{\delta}{1-\delta}\frac{\rho+\left(\frac{1}{\beta}+\frac{1}{4\rho\beta^2\alpha}\right)\frac{1}{\delta}}{\frac{1}{\beta}+\frac{1}{4\rho\beta^2\alpha}}\\
\leq & \ \ul \lambda \frac{2\rho(1-\alpha)}{\rho+\left(\frac{1}{\beta}+\frac{1}{4\rho\beta^2\alpha}\right)\frac{1}{\delta}}
+\ul \lambda^2 \left(\frac{2\rho(1-\alpha)}{\rho+\left(\frac{1}{\beta}+\frac{1}{4\rho\beta^2\alpha}\right)\frac{1}{\delta}}\right)^2
\end{align*}
and \begin{equation*}
\frac{1+\sqrt{1+4\theta}}{2}\leq\gamma \leq 1+\ul \lambda \frac{2\rho(1-\alpha)}{\rho+\left(\frac{1}{\beta}+\frac{1}{4\rho\beta^2\alpha}\right)\frac{1}{\delta}}.
\end{equation*}
\end{remark}

When considering the convex optimization problem \eqref{convopt}, the second order dynamical system \eqref{s-o-dyn-syst-fb} written for $A=\partial f$ and $B= \nabla g$ becomes \begin{equation}\label{s-o-dyn-syst-fb-opt}\left\{
\begin{array}{ll}
\ddot x(t) + \gamma(t) \dot x(t) + \lambda(t)\left[x(t)-\prox_{\eta f}\Big(x(t)-\eta \nabla g(x(t))\Big)\right]=0\\
x(0)=u_0, \dot x(0)=v_0.
\end{array}\right.\end{equation}
Theorem \ref{sec-ord-a+b} gives rise to the following result.

\begin{theorem}\label{sec-ord-opt} Let $f:{\cal H}\rightarrow \R\cup\{+\infty\}$ be a proper, convex 
and lower semicontinuous function, $g:{\cal H}\rightarrow \R$ be a convex and (Fr\'{e}chet) differentiable function with $\frac{1}{\beta}$-Lipschitz continuous gradient for $\beta > 0$ such that 
$f+g$ is $\rho$-strongly convex for $\rho>0$ and  $x^*$ be the unique minimizer of $f+g$ over ${\cal H}$. Chose $\alpha,\delta\in(0,1)$ and $\eta > 0$ such that $\delta\beta\rho<1$ and 
$\frac{1}{\eta}=\left(\frac{1}{\beta}+\frac{1}{4\rho\beta^2\alpha}\right)\frac{1}{\delta} - \rho >0$. 

Let  $\lambda:[0,+\infty)\rightarrow[0,+\infty)$ be a locally absolutely continuous 
function fulfilling for every $t \in [0,+\infty)$
\begin{align*}\theta(t):= & \ \lambda(t)\frac{\delta}{1-\delta}\frac{\rho+\left(\frac{1}{\beta}+\frac{1}{4\rho\beta^2\alpha}\right)\frac{1}{\delta}}{\frac{1}{\beta}+\frac{1}{4\rho\beta^2\alpha}}\\
\leq & \ \lambda(t)\frac{2\rho(1-\alpha)}{\rho+\left(\frac{1}{\beta}+\frac{1}{4\rho\beta^2\alpha}\right)\frac{1}{\delta}}
+\lambda^2(t)\left(\frac{2\rho(1-\alpha)}{\rho+\left(\frac{1}{\beta}+\frac{1}{4\rho\beta^2\alpha}\right)\frac{1}{\delta}}\right)^2
\end{align*}
and such that there exists a real number $\ul\lambda$ with the property that
$$0<\ul\lambda\leq\inf_{t\geq0}\lambda(t)$$ and 
$$2<\theta:=\ul\lambda\frac{\delta}{1-\delta}\frac{\rho+\left(\frac{1}{\beta}+\frac{1}{4\rho\beta^2\alpha}\right)\frac{1}{\delta}}{\frac{1}{\beta}+\frac{1}{4\rho\beta^2\alpha}}.$$

Further, let $\gamma:[0,+\infty)\rightarrow[0,+\infty)$ be a locally absolutely continuous 
function fulfilling \eqref{th6cond1} and \eqref{th6cond2}, $u_0,v_0\in {\cal H}$  and $x:[0,+\infty)\rightarrow{\cal H}$ be the unique strong global solution of the dynamical system \eqref{s-o-dyn-syst-fb-opt}.

Then $\gamma(t)\geq\ul\gamma:=\frac{1+\sqrt{1+4\theta}}{2}>2$ for every $t \in [0,+\infty)$ and there exists $M>0$ such that for every $t\in[0,+\infty)$
\begin{align*}
0\leq \|x(t)-x^*\|^2 & \leq \|u_0-x^*\|^2\exp(-(\ul\gamma-1)t)+\frac{M}{\ul\gamma-2}\exp(-t)\\
 & \leq \left(\|u_0-x^*\|^2+\frac{M}{\ul\gamma-2}\right)\exp(-t).
\end{align*}
\end{theorem}

Finally, we approach the convex minimization problem \eqref{gopt} via the second order dynamical system
\begin{equation}\label{s-o-nabla}\left\{
\begin{array}{ll}
\ddot x(t) + \gamma(t) \dot x(t) + \lambda(t)\nabla g(x(t))=0\\
x(0)=u_0, \dot x(0)=v_0
\end{array}\right.\end{equation}
and provide an exponential rate of convergence of $g$ to its minimum value along the generated trajectories. 
The following result can be seen as the continuous counterpart of \cite[Theorem 4]{gfj}, where recently a linear rate of convergence for the values of $g$ on a sequence iteratively generated by an inertial-type algorithm has been obtained.  

\begin{theorem}\label{sec-ord-nabla} Let $g:{\cal H}\rightarrow \R$ be a $\rho$-strongly convex and 
(Fr\'echet) differentiable function with $\frac{1}{\beta}$-Lipschitz continuous gradient for $\rho > 0$ and $\beta > 0$ and $x^*$ be the unique minimizer of $g$ over ${\cal H}$.

Let $\alpha:[0,+\infty)\rightarrow\R$ be a Lebesgue measurable function such that there exists $\ul \alpha >1$ with
\begin{equation}\label{ulaplha}\inf_{t\geq 0}\alpha(t) \geq \max \left\{\ul \alpha, \frac{2}{\beta^2\rho^2} - 1\right\} \end{equation} 
and $\lambda:[0,+\infty)\rightarrow[0,+\infty)$ be a locally absolutely continuous function
fulfilling for every $t \in [0,+\infty)$
\begin{equation}\label{s-o-nabla-lambda}\frac{\alpha(t)}{\beta\rho^2}\leq \lambda(t)\leq \frac{\beta}{2}\big(\alpha(t)+\alpha^2(t)\big).\end{equation}
Further, let $\gamma:[0,+\infty)\rightarrow[0,+\infty)$ be a locally absolutely continuous 
function fulfilling   
\begin{equation}\label{s-o-nabla-gamma}\frac{1+\sqrt{1+8\frac{\lambda(t)}{\beta}}}{2}\leq\gamma(t)\leq 1+\alpha(t) \ \mbox{for every} \ t \in [0,+\infty) \end{equation} and \eqref{th6cond2}.

Let $u_0,v_0\in {\cal H}$ and $x:[0,+\infty)\rightarrow{\cal H}$ be the unique strong global solution of  the dynamical system \eqref{s-o-nabla}.

Then $\gamma(t)\geq\ul\gamma:=\frac{1+\sqrt{1+8\frac{\ul\alpha}{\beta^2\rho^2}}}{2}>2$ and there exists $M>0$ such that for every $t \in [0,+\infty)$ 
\begin{align*}
0\leq\frac{\rho}{2}\|x(t)-x^*\|^2 & \leq g(x(t))-g(x^*)\leq (g(u_0)-g(x^*))\exp(-(\ul\gamma-1)t)+\frac{M}{\ul\gamma-2}\exp(-t)\\
& \leq \! \left(\!g(u_0)-g(x^*)+\frac{M}{\ul\gamma-2}\!\right)\!\exp(-t) \! \leq \! \left(\! \frac{1}{2\beta}\|u_0-x^*\|^2+\frac{M}{\ul\gamma-2}\!\right)\!\exp(-t).
\end{align*}
\end{theorem}

\begin{proof} One has for almost every $t \in [0,+\infty)$
$$\frac{d}{dt}g(x(t))=\<\dot x(t),\nabla g(x(t))\>$$ and (see Remark \ref{rem-abs-cont}(b))
$$\frac{d^2}{dt^2}g(x(t))=\<\ddot x(t),\nabla g(x(t))\>+
\<\dot x(t),\frac{d}{dt}\nabla g(x(t))\>\leq \<\ddot x(t),\nabla g(x(t))\>+
\frac{1}{\beta}\|\dot x(t)\|^2.$$
Further, by using  \eqref{desc-l}, \eqref{ineq-g-str-conv} and the first equation in \eqref{s-o-nabla}, we derive for almost every $t \in [0,+\infty)$
\begin{align*}& \ \frac{d^2}{dt^2}\big(g(x(t))-g(x^*)\big)+\gamma(t)\frac{d}{dt}\big(g(x(t))-g(x^*)\big)+\alpha(t)\big(g(x(t))-g(x^*)\big)\\
 \leq & \ -\lambda(t)\|\nabla g(x(t))\|^2+\frac{\alpha(t)}{2\beta\rho^2}\|\nabla g(x(t))\|^2+\frac{1}{\beta}\|\dot x(t)\|^2\\
 = & \ -\frac{1}{2\lambda(t)}\|\ddot x(t)+\gamma(t)\dot x(t)\|^2-\frac{\lambda(t)}{2}\|\nabla g(x(t))\|^2+\frac{\alpha(t)}{2\beta\rho^2}\|\nabla g(x(t))\|^2+\frac{1}{\beta}\|\dot x(t)\|^2.
\end{align*}
 
Taking into account \eqref{norm^2} we obtain  for almost every $t \in [0,+\infty)$
\begin{align*}& \ \frac{d^2}{dt^2}\big(g(x(t))-g(x^*)\big)+\gamma(t)\frac{d}{dt}\big(g(x(t))-g(x^*)\big)+\alpha(t)\big(g(x(t))-g(x^*)\big)\\
+ & \ \frac{\gamma(t)}{2\lambda(t)}\frac{d}{dt}(\|\dot x(t)\|^2)+\left(\frac{\gamma^2(t)}{2\lambda(t)}-\frac{1}{\beta}\right)\|\dot x(t)\|^2 \\
+ & \ \frac{1}{2\lambda(t)}\|\ddot x(t)\|^2+\left(\frac{\lambda(t)}{2}-\frac{\alpha(t)}{2\beta\rho^2}\right)\|\nabla g(x(t))\|^2 \leq 0.
\end{align*}
According to the choice of the parameters involved, we have 
$$\frac{\lambda(t)}{2}-\frac{\alpha(t)}{2\beta\rho^2}\geq 0,$$
thus, for almost every $t \in [0,+\infty)$, 
\begin{align*}& \ \frac{d^2}{dt^2}\big(g(x(t))-g(x^*)\big)+\gamma(t)\frac{d}{dt}\big(g(x(t))-g(x^*)\big)+\alpha(t)\big(g(x(t))-g(x^*)\big)\\
+ & \ \frac{\gamma(t)}{2\lambda(t)}\frac{d}{dt}(\|\dot x(t)\|^2)+\left(\frac{\gamma^2(t)}{2\lambda(t)}-\frac{1}{\beta}\right)\|\dot x(t)\|^2 \leq 0.
\end{align*}

This shows that \eqref{ineq-h-u} in Lemma \ref{l-sec-ord} for $u := \|\dot x(\cdot)\|^2$,
$$b_1(t):=\alpha(t),$$
$$b_2(t):=  \frac{\gamma(t)}{2\lambda(t)}$$
and
$$b_3(t):=\frac{\gamma^2(t)}{2\lambda(t)}-\frac{1}{\beta}$$
is fulfilled. 
By combining \eqref{s-o-nabla-gamma} and the first condition in \eqref{th6cond2} one obtains \eqref{hyp1}, while, by combining \eqref{s-o-nabla-gamma} and the second condition in \eqref{th6cond2} one obtains \eqref{hyp2}.

Furthermore, by taking into account the Lipschitz property of $\nabla g$ and the strong convexity of $g$, it yields 
$$\rho\beta\leq 1.$$
From \eqref{s-o-nabla-lambda}, \eqref{ulaplha} and $\ul\alpha>1$ we obtain $$\frac{\lambda(t)}{\beta}\geq\ul\alpha\frac{1}{\beta^2\rho^2}>1 \ \mbox{for every} \ t \in [0,+\infty),$$
which combined with \eqref{s-o-nabla-gamma} leads to $\ul\gamma>2$. 

The conclusion follows  from Lemma \ref{l-sec-ord}(ii), the strong convexity of $g$ and 
\eqref{desc-l}. 
\end{proof}

\begin{remark}\label{rem32}
In Theorem \ref{sec-ord-nabla} one can obviously chose $\alpha(t) = \alpha$, where $\alpha = \frac{2}{\beta^2\rho^2} - 1$, if $\beta\rho < 1$, or $\alpha = 1+ \varepsilon$, with $\varepsilon >0$, otherwise, $\lambda(t) = \lambda$ and $\gamma(t) = \gamma$ for every $t \in [0,+\infty)$, where
\begin{equation*}
\frac{\alpha}{\beta\rho^2}\leq \lambda \leq \frac{\beta}{2}\big(\alpha +\alpha^2 \big)\end{equation*}
and
\begin{equation*}
\frac{1+\sqrt{1+8\frac{\lambda}{\beta}}}{2}\leq\gamma \leq 1+\alpha. \end{equation*}
\end{remark}


\begin{thebibliography}{99}

\bibitem{abbas-att-arx14} B. Abbas, H. Attouch, {\it Dynamical systems and forward-backward algorithms associated 
with the sum of a convex subdifferential and a monotone cocoercive operator}, Optimization, DOI: 10.1080/02331934.2014.971412, 2014

\bibitem{abbas-att-sv} B. Abbas, H. Attouch, B.F. Svaiter, {\it Newton-like dynamics and forward-backward methods for 
structured monotone inclusions in Hilbert spaces}, Journal of Optimization Theory and its Applications 161(2), 331--360, 2014

\bibitem{alvarez2000} F. Alvarez, {\it On the minimizing property of a second order dissipative system in Hilbert spaces}, SIAM Journal
on Control and Optimization 38(4), 1102--1119, 2000

\bibitem{alvarez2004} F. Alvarez, {\it Weak convergence of a relaxed and inertial hybrid projection-proximal point algorithm for maximal monotone operators in Hilbert space}, SIAM Journal on Optimization 14(3), 773--782, 2004

\bibitem{alv-att-sva} F. Alvarez, H. Attouch, {\it An inertial proximal method for maximal monotone operators via discretization of  a nonlinear oscillator with damping}, Set-Valued Analysis 9(1-2), 3–11, 2001

\bibitem{alv-att-bolte-red} F. Alvarez, H. Attouch, J. Bolte, P. Redont, {\it A second-order gradient-like dissipative dynamical system 
with Hessian-driven damping. Application to optimization and mechanics}, Journal de Math\'{e}matiques Pures et Appliqu\'{e}es (9) 81(8), 
747--779, 2002

\bibitem{antipin} A.S. Antipin, {\it Minimization of convex functions on convex sets by means of differential equations}, 
(Russian) Differentsial'nye Uravneniya 30(9), 1475--1486, 1994; translation in Differential Equations 30(9), 1365--1375, 1994

\bibitem{att-alv} H. Attouch, F. Alvarez, {\it The heavy ball with friction dynamical system for convex constrained 
minimization problems}, in: Optimization (Namur, 1998), 25--35, in: Lecture Notes in Economics and Mathematical Systems 481, Springer, Berlin, 2000

\bibitem{att-alv-sv} H. Attouch, M. Marques Alves, B.F. Svaiter, {\it A dynamic approach to a proximal-Newton method for monotone inclusions 
in Hilbert spaces, with complexity $O(1/{n^2})$}, arXiv:1502.04286v1    

\bibitem{att-cza-10} H. Attouch, M.-O. Czarnecki, {\it Asymptotic behavior of coupled dynamical systems with multiscale aspects}, 
Journal of Differential Equations 248(6), 1315-1344, 2010

\bibitem{att-g-r} H. Attouch, X. Goudou, P. Redont, {\it The heavy ball with friction method. I. The continuous dynamical system: 
global exploration of the local minima of a real-valued function by asymptotic analysis of a dissipative dynamical system}, 
Communications in Contemporary Mathematics 2(1), 1--34, 2000

\bibitem{att-sv2011} H. Attouch, B.F. Svaiter, {\it A continuous dynamical Newton-like approach to solving monotone inclusions}, 
SIAM Journal on Control and Optimization 49(2), 574--598, 2011

\bibitem{bauschke-book} H.H. Bauschke, P.L. Combettes, {\it Convex Analysis and Monotone Operator Theory in Hilbert Spaces}, CMS Books in Mathematics, Springer, New York, 2011

\bibitem{bolte-2003} J. Bolte, {\it Continuous gradient projection method in Hilbert spaces}, Journal of Optimization Theory and its 
Applications 119(2), 235--259, 2003

\bibitem{b-c-dyn-KM} R.I. Bo\c t, E.R. Csetnek, {\it A dynamical system associated with the fixed points set of a 
nonexpansive operator}, Journal of Dynamics and Differential Equations, DOI: 10.1007/s10884-015-9438-x, 2015

\bibitem{b-c-dyn-pen} R.I. Bo\c t, E.R. Csetnek, {\it Approaching the solving of constrained variational inequalities via penalty 
term-based dynamical systems}, arXiv:1503.01871

\bibitem{b-c-dyn-sec-ord} R.I. Bo\c t, E.R. Csetnek, {\it Second order forward-backward dynamical systems for
monotone inclusion problems}, arXiv:1503.04652

\bibitem{b-c-h-inertial} R.I. Bo\c t, E.R. Csetnek, C. Hendrich, {\it Inertial Douglas-Rachford splitting for monotone inclusion problems}, 
Applied Mathematics and Computation 256, 472--487, 2015

\bibitem{gfj} E. Ghadimi, H.R. Feyzmahdavian, M. Johansson, {\it Global convergence of the Heavy-ball method for convex optimization}, 
arXiv:1412.7457 

\bibitem{nes} Y. Nesterov, {\it Introductory Lectures on Convex Optimization: A Basic Course}, Kluwer Academic Publishers, Dordrecht, 2004

\bibitem{pp} J.-C. Pesquet, N. Pustelnik, {\it A parallel inertial proximal optimization method}, 
Pacific Journal of Optimization 8(2), 273--305, 2012

\bibitem{rock} R.T. Rockafellar, {\it On the maximal monotonicity of subdifferential mappings}, Pacific Journal of
Mathematics 33(1), 209--216, 1970

\bibitem{rock-prox} R.T. Rockafellar, {\it Monotone operators and the proximal point algorithm}, SIAM Journal on Control and Optimization 14(5), 877--898, 1976


\end{thebibliography}
\end{document}